\documentclass[10pt]{amsart}
\usepackage{amsmath}
\usepackage{amssymb}
\usepackage{amsthm}
\usepackage{graphicx}

 \newtheorem{theorem}{Theorem}[section]
 \newtheorem{corollary}[theorem]{Corollary}
 \newtheorem{proposition}[theorem]{Proposition}
 \newtheorem{lemma}[theorem]{Lemma}

 {\theoremstyle{definition}
 }
 {\theoremstyle{definition}
  \newtheorem{definition}[theorem]{Definition}}
 {\theoremstyle{definition}
 \newtheorem{example}[theorem]{Example}}
 {\theoremstyle{definition}
 }
  {\theoremstyle{definition}
 }  
 {\theoremstyle{definition}
 }

 \newcommand{\LO}{\textrm{LO}}
 \newcommand{\Z}{\mathbb{Z}}

 \newcommand{\mG}{\mathcal{G}}
 \newcommand{\mF}{\mathcal{F}}
 \newcommand{\mX}{\mathcal{X}}
 \newcommand{\mS}{\mathcal{S}}

  \newcommand{\mH}{\mathcal{H} }

\title[Construction of isolated left orderings]{Construction of isolated left orderings via partially central cyclic amalgamation}
\author{Tetsuya Ito}
\address{Department of Mathematics,
The University of British Columbia, 1984 Mathematics Road
Vancouver, B.C, Canada V6T 1Z2}
\email{tetitoh@math.ubc.ca}

\begin{document}
\maketitle

\begin{abstract}
We give a new method to construct isolated left orderings of groups whose positive cones are finitely generated. Our construction uses an amalgamated free product of two groups having an isolated ordering. We construct a lot of new examples of isolated orderings, and give an example of isolated left orderings having various properties which previously known isolated orderings do not have.  
\end{abstract}

\section{Introduction}

A total ordering $<_{G}$ on a group $G$ is a {\em left ordering} if $g <_{G} g'$ implies $hg<_{G} hg'$ for all $g,g',h \in G$. 
The {\em positive cone} of a left ordering $<_{G}$ is a sub-semigroup $P(<_{G})$ of $G$ consisting of $<_{G}$-positive elements.

The set of all left orderings of $G$ is denoted by $\LO(G)$. For $g \in G$, let $U_{g}$ be a subset of $\LO(G)$ defined by  
\[ U_{g}= \{ <_{G} \,\in \LO(G) \: | \: 1<_{G} g\}. \]
We equip a topology on $\LO(G)$ so that $\{U_{g}\}_{g \in G}$ is an open sub-basis of the topology.
This topology is understood as follows. For a left ordering $<_{G}$ of $G$, $G$ is decomposed as a disjoint union $G= P(<_{G}) \sqcup \{1\} \sqcup P(<_{G})^{-1}$ using the positive cone $P(<_{G})$. Conversely, a sub-semigroup $P$ of $G$ having this property is a positive cone of a left ordering of $G$: An ordering $<_{P}$ defined by  $g<_{P} g'$ if $g^{-1}g' \in P$ is a left-ordering whose positive cone is $P$. Thus $\LO(G)$ is identified with a subset of the powerset $2^{G -\{1\}}$. The topology of $\LO(G)$ defined as above coincides with the relative topology as the subspace of $2^{G -\{1\}}$, equipped with the powerset topology.

In this paper, we always consider {\em countable} groups, so we simply refer a countable group as a group unless otherwise specified. 
Then it is known that $\LO(G)$ is a compact, metrizable, and totally disconnected \cite{s}. Moreover, $\LO(G)$ is either uncountable or finite \cite{l}. 
Thus as a topological space, $\LO(G)$ is rather similar to the Cantor set: The main difference is that the space $\LO(G)$ might be non-perfect, that is, $\LO(G)$ might have isolated points. Indeed, if $\LO(G)$ has no isolated points and is not a finite set, then $\LO(G)$ is homeomorphic to the Cantor set. We call a left ordering which is an isolated point of $\LO(G)$ an {\em isolated ordering}.

 It is known that a left ordering $<_{G}$ whose positive cone is a finitely generated semigroup is isolated. In this paper we will concentrate our attention to study such an isolated ordering. We say a finite set of non-trivial elements of $G$, $\mG=\{g_{1},\ldots,g_{r}\}$ {\em defines} an isolated left ordering $<_{G}$ of $G$ if the positive cone of $<_{G}$ is generated by $\mG$ as a semigroup. For an isolated left ordering $<_{G}$ of a group $G$, the {\em rank} of $<_{G}$ is the minimal number of generating sets of the positive cone of $<_{G}$ and denoted by $r(<_{G})$. (If $P(<_{G})$ is not finitely generated semigroup we define $r(<_{G}) = \infty$).

We say an isolated ordering $<_{G}$ of $G$ is {\em genuine} if $\LO(G)$ is not a finite set. Then $\LO(G)$ contains (uncountably many) non-isolated points. 
The classification of groups having non-genuine isolated orderings, namely, the classification of groups having finitely many left-orderings is given by Tararin (see \cite{km}). 
On the other hand, it is difficult to construct genuine isolated left orderings, and few examples are known. At this time of writing, as long as author's knowledge, there are essentially only two families of genuine isolated left orderings.\\

\begin{enumerate}
\item {\em Dubrovina-Dubrovin ordering} \cite{ddrw},\cite{dd}.\\

Let $\sigma_{1},\ldots,\sigma_{n-1}$ be the standard generator of the $n$-strand braid group $B_{n}$. 
The {\em Dubrovina-Dubrovin ordering} $<_{DD}$ is an isolated left ordering of $B_{n}$ whose positive cone is generated by $\{a_{1},\ldots,a_{n-1}\}$, where $a_{i}$ is given by 
\[ a_{i}= (\sigma_{n-i}\sigma_{n-i+1} \cdots \sigma_{n-1})^{(-1)^{i}}\]
The rank of the Dubrovina-Dubrovin ordering $<_{DD}$ is $n-1$. See \cite{ddrw}, \cite{dd} for details.\\

\item {\em Isolated orderings of} $\Z*_{\Z}\Z$ \cite{i1},\cite{n2}.\\

Let $G = \Z*_{\Z} \Z$ be the group obtained as an amalgamated free product of two infinite cyclic groups over $\Z$. Thus, $G$ is presented as 
\[ G = \langle x,y \: | \: x^{m}= y^{n} \rangle. \]
by using some integers $m$ and $n$. Then the generating set $\{xy^{1-n},y\}$ defines an isolated left ordering $<_{A}$ of $G$, which is genuine if  $(m,n) \neq (2,2)$. The rank of $<_{A}$ is $2$.
This example was found by Navas \cite{n2} for the case $m=2$, and by the author \cite{i1} for general cases.
We here remark that if $(m,n)=(2,3)$ then $G_{m,n}$ is the 3-braid group $B_{3}$, and the isolated ordering $<_{A}$ is the same as the Dubrovina-Dubrovin ordering $<_{DD}$. \\

\end{enumerate}

Thus, it is desirable to find more examples or general constructions of isolated left orderings. 

In author's previous paper \cite{i1}, we gave one general method to construct isolated orderings by using rather combinatorial approach.
Following \cite{n2}, we introduced a notion of {\em Dehornoy-like ordering}. This is a left-ordering whose positive cone consists of certain kind of words over a special generating set $\mS$ of $G$, which we called {\em $\sigma(\mS)$-positive words}. A Dehornoy-like ordering is a generalization of the Dehornoy ordering of the braid groups, one of the most interesting left orderings: The Dehornoy ordering has various stimulating features and a lot of interesting interpretations that relate many aspects of braid groups and orderings. See \cite{ddrw} for the theory of the Dehornoy ordering. One fascinating property of the Dehornoy ordering is that one get the Dubrovina-Dubrovin ordering by modifying the Dehornoy ordering.

We showed that, under some condition which we called the Property $F$, Dehornoy-like orderings and the Dehornoy ordering share various properties. In particular, we have shown that a Dehornoy-like ordering produces an isolated ordering and vice versa. Indeed, it is shown that the above two families of known isolated orderings are derived from Dehornoy-like orderings.

However it seems to be more difficult to find an example of a Dehornoy-like ordering than to find an example of an isolated ordering directly, since the definition of Dehornoy-like orderings includes complicated combinatorics.

The aim of this paper is to give a new construction of isolated left orderings by means of the {\em partially central cyclic amalgamation}. 
From two groups having (not necessarily genuine) isolated orderings, we construct a new group having an isolated left ordering by using amalgamated free product over $\Z$.

 In almost all cases, the obtained isolated orderings are genuine. Our construction can be seen as an extension of (2) of known examples, but it is completely different from the Dehornoy-like orderings construction. In fact, we will see that some of the so far orderings cannot be obtained from Dehornoy-like orderings.

The following is the summary of main result of this paper.
Recall that for $g \in G$ and a left ordering $<_{G}$ of $G$, $<_{G}$ called a {\em $g$-right invariant ordering} if the ordering $<_{G}$ is preserved by the right multiplication of $g$, that is, $a<_{G} b$ implies $ag <_{G} bg$ for all $a,b \in G$.

\begin{theorem}[Construction of isolated left ordering via partially central cyclic amalgamation]
\label{thm:main}
Let $G$ and $H$ be finitely generated groups. 
Let $z_{G}$ be a non-trivial central element of $G$, and $z_{H}$ be a non-trivial element of $H$.

Let $\mG =\{g_{1},\ldots,g_{m}\}$ be a finite generating set of $G$ which defines an isolated left ordering $<_{G}$ of $G$. We take a numbering of elements of $\mG$ so that $1<_{G} g_{1} <_{G} \cdots <_{G} g_{m}$ holds. 
Similarly, let $\mH=\{h_{1},\ldots,h_{n}\}$ be a finite generating set of $H$ which defines an isolated left ordering $<_{H}$ of $H$ such that the inequalities $1<_{H} h_{1} <_{H} \cdots <_{H} h_{n}$ hold.

We assume the cofinality assumptions {\bf [CF(G)]}, {\bf [CF(H)]}, and the invariance assumption {\bf [INV(H)]}.
\begin{align*}
\bf{[CF(G)]} & \quad g_{i} <_{G} z_{G} \text{ holds for all } i.\\
\bf{[CF(H)]} & \quad h_{i} <_{H} z_{H} \text{ holds for all } i.\\
\bf{[INV(H)]} & \quad <_{H} \text{ is a } z_{H}\text{-right invariant ordering}.
\end{align*}

Let $X= G * _{\Z} H = G *_{\langle z_{G}=z_{H} \rangle} H$ be an amalgamated free product of $G$ and $H$ over $\Z$. For $i=1,\ldots, m$, let $x_{i}=g_{i}z_{H}^{-1}h_{1}$. Then we have the following results:

\begin{enumerate}
\item[(i)] The generating set $\{x_{1},\ldots, x_{m}, h_{1},\ldots, h_{n}\}$ of $X$ defines an isolated left ordering $<_{X}$ of $X$.
\item[(ii)] The isolated ordering $<_{X}$ does not depend on the choice of the generating sets $\mG$ and $\mH$. Thus, $<_{X}$ only depends on the isolated orderings $<_{G},<_{H}$ and the elements $z_{G},z_{H}$.
\item[(iii)] The natural inclusions $\iota_{G} : G \rightarrow X$ and $\iota_{H}: H \rightarrow X$ are order-preserving homomorphisms.
\item[(iv)] $1<_{X} x_{1} <_{X} \cdots <_{X} x_{m} <_{X} h_{1} <_{X} \cdots <_{X} h_{n} <_{X} z_{H}=z_{G}$. Moreover, $z= z_{G}=z_{H}$ is $<_{X}$-positive cofinal and the isolated ordering $<_{X}$ is a $z$-right invariant ordering.
\item[(v)] $r(<_{X}) \leq r(<_{G}) + r(<_{H})$.
\item[(vi)] Let $Y$ be a non-trivial proper subgroup of $X$. If $Y$ is $<_{X}$-convex, then $Y= \langle x_{1} \rangle$, the infinite cyclic group generated by $x_{1}$. 
\end{enumerate}
\end{theorem}

We call the construction of isolated ordering described in Theorem \ref{thm:main} the {\em partially central cyclic amalgamation construction}. 

As we will see in Lemma \ref{lem:cofinal} in Section \ref{sec:assump}, the cofinality assumption {\bf CF(G)} (resp. {\bf CF(H)}) is understood as an assumption on $z_{G}$ and $<_{G}$ (resp. $z_{H}$ and $<_{H}$). Thus, Theorem \ref{thm:main} (ii) shows that the choice of the generating sets $\mG$ and $\mH$ is not important, though it is useful to describe and understand the isolated ordering $<_{X}$. The generating sets $\mG$ and $\mH$ play rather auxiliary roles and are not essential in our partially central cyclic amalgamation construction. This makes a sharp contrast with the construction using Dehornoy-like orderings, since in the Dehornoy-like ordering construction we need to use a special generating set derived from Dehornoy-like ordering having the nice property which we called the Property F.

Theorem \ref{thm:main} (iii) shows that the partially central cyclic amalgamation construction can be seen as a mixing of two isolated orderings $<_{G}$ and $<_{H}$. We remark that Theorem \ref{thm:main} (iv) ensures that we can iterate the partially central cyclic amalgamation construction. Thus, we can actually produce many isolated orderings by using the partially central cyclic amalgamation constructions from known isolated orderings.

The proof of Theorem \ref{thm:main} (i) is constructive and will actually provide an algorithm to determine the isolated ordering $<_{X}$. In particular, the isolated ordering $<_{X}$ can be determined algorithmically if we have algorithms to compute the isolated orderings $<_{G}$ and $<_{H}$, as we will see in Section \ref{sec:comp}.

It is interesting to compare our amalgamation with other natural operations on groups.
Unlike the partial central cyclic amalgamation (the amalgamated free product over $\Z$, used in Theorem \ref{thm:main}), the usual free product does not preserve the property that the group has an isolated left ordering.
Here is the simplest counter-example: the free group of rank two $F_{2}=\Z*\Z$ has no isolated orderings \cite{n1}, whereas the infinite cyclic group $\Z$ has (non-genuine) isolated orderings, since it admits only two left orderings. Indeed, recently Rivas \cite{r} proved that free products of groups do not have any isolated left orderings. Similarly, the direct products of groups also do not preserve the property that the group has an isolated left ordering: the free abelian group of rank two $\Z \times \Z$ has no isolated orderings \cite{s}.

The plan of this paper is as follows: In Section \ref{sec:const} we prove Theorem \ref{thm:main}. The main technical tool of the proof is a {\em reduced standard factorization}, which serves as some kind of normal form of elements in $X$, adapted to the generating set $\{x_{1},\ldots,x_{m},h_{1},\ldots,h_{n}\}$. In Section \ref{sec:exam} we give some examples of isolated orderings obtained by applying Theorem \ref{thm:main}. We observe that our examples have various interesting properties, which have not been appeared for known examples.

\section{Construction of isolated left orderings}
\label{sec:const}

Let $\mS= \{s_{1},\ldots,s_{n}\}$ be a finite generating set of $G$ and let $\mS^{-1}=\{s_{1}^{-1},\ldots,s_{n}^{-1}\}$. 
We denote by $\mS^{*}$ the free semigroup generated by $\mS$. That is, $\mS^{*}$ is a set of non-empty words on $\mS$.
We say an element of $\mS^{*}$ (resp. $(\mS^{-1})^{*}$) an {\em $\mS$-positive word} (resp. an {\em $\mS$-negative word}). We will often use a symbol $P(\mS)$ (resp. $N(\mS)$) to represent an $\mS$-positive (resp.  $\mS$-negative) word.

\subsection{Cofinality and Invariance assumptions}
\label{sec:assump}

First of all we review the assumptions in the statement of Theorem \ref{thm:main} again, and deduce their direct consequences. This clarifies the role of each hypothesis in Theorem \ref{thm:main}.

Let $G$ and $H$ be finitely generated groups having an isolated left ordering $<_{G}$ and $<_{H}$ respectively.
Let $z_{G} \in G$ be a non-trivial central element of $G$, and let $z_{H}$ be a non-trivial element of $H$, which might be noncentral. We consider the group $X$ obtained as an amalgamated free product over $\Z$
\[ X = G*_{\Z} H = G*_{\langle z_{G}=z_{H} \rangle} H.\]

Let $\mG =\{g_{1},\ldots,g_{m}\}$ be a generating set of $G$ which defines an isolated left ordering $<_{G}$ of $G$. We take a numbering of elements of $\mG$ so that $1 <_{G} g_{1} <_{G} \cdots <_{G} g_{m}$ holds. 
Similarly, let $\mH=\{h_{1},\ldots,h_{n}\}$ be a generating set of $H$ which defines an isolated left ordering $<_{H}$ of $H$, and we assume that the inequalities $1<_{H} h_{1} <_{H} \cdots <_{H} h_{n}$ hold.

Recall that an element $g \in G$ is called the {\em $<_{G}$-minimal positive element} if $g$ is the $<_{G}$-minimal element in the positive cone $P(<_{G})$.  In other words, the inequality $1 <_{G} g' \leq_{G} g$ implies $g=g'$. A left ordering $<_{G}$ is called {\em discrete} if $<_{G}$ has a $<_{G}$-minimal positive element. Otherwise, $<_{G}$ is called {\em dense}.

 As the next lemma shows, the choice of the numbering of $\mG$ (resp. $\mH$) implies that $g_{1}$ (resp. $h_{1}$) is the $<_{G}$-minimal (resp. $<_{H}$-minimal) positive element. In particular, $g_{1}$ (resp. $h_{1}$) is independent of the choice of the generating set $\mG$ (resp. $\mH$).

\begin{lemma}
\label{lem:minimal}
Let $\mG=\{g_{1},\ldots, g_{m}\}$ be a generating set of a group $G$ which defines an isolated left ordering $<_{G}$ of $G$.
Assume that $g_{1}$ is the $<_{G}$-minimal element in the set $\mG$. Then $g_{1}$ is the $<_{G}$-minimal positive element. In particular, $<_{G}$ is discrete. Moreover, $<_{G}$ is a $g_{1}$-right invariant ordering.
\end{lemma}
\begin{proof}
Assume $g \in G$ satisfies the inequalities $1<_{G} g \leq_{G} g_{1}$. 

$1<_{G} g$ means that $g$ is written as a $\mG$-positive word $g=g_{i_{1}}\cdots g_{i_{l}}$. Then 
\[ g_{1}^{-1}g = (g_{1}^{-1}g_{i_{1}}) g_{i_{2}} \cdots g_{i_{l}} \leq_{G} 1\]
This inequality holds only if $i_{1}=1$ and $l=1$, that is, $g=g_{1}$.

The $g_{1}$-right invariance of the ordering $<_{G}$ now follows from the fact that $g_{1}$ is the $<_{G}$-minimal positive element: If $a <_{G}b$, then $1<_{G} a^{-1}b <_{G} a^{-1}bg_{1}$. Thus, $g_{1} <_{G} a^{-1}bg_{1}$ so $ag_{1} <_{G} bg_{1}$.
\end{proof}

To obtain an isolated ordering of $X$ from $<_{G}$ and $<_{H}$, we impose the following assumptions, which we call the {\em cofinality assumption} for $G$ and $H$, and the {\em invariance assumption}.
\begin{align*}
\bf{[CF(G)]} & \quad g_{i} <_{G} z_{G} \text{ holds for all } i.\\
\bf{[CF(H)]} & \quad h_{i} <_{H} z_{H} \text{ holds for all } i.\\
\bf{[INV(H)]} & \quad <_{H} \text{ is a } z_{H}\text{-right invariant ordering}.
\end{align*}

Here we remark that the invariance assumption for $<_{G}$ is automatically satisfied: that is, $<_{G}$ is a $z_{G}$-right invariant ordering since we have chosen $z_{G}$ so that it is a central element.

First we observe the following simple lemma. 

\begin{lemma}
\label{lem:commute}
Let $<_{H}$ be a discrete left ordering of a group $H$, and let $h_{1}$ be the $<_{H}$-minimal positive element.
If $<_{H}$ is an $h$-right invariant ordering for $h \in H$, then $h$ commutes with $h_{1}$.
\end{lemma}
\begin{proof}
$<_{H}$ is an $h$-right invariant ordering, so $hh_{1}h^{-1} >_{H} 1$ and $h^{-1}h_{1}h >_{H} 1 $.  $h_{1}$ is the $<_{H}$-minimal positive element, so $hh_{1}h^{-1} \geq_{H} h_{1}$ and $h^{-1}h_{1}h \geq_{H} h_{1}$. 
Thus, we get $hh_{1}\geq_{H} h_{1}h$ and $h_{1}h \geq_{H} hh_{1}$, hence $hh_{1}=h_{1}h$.
\end{proof}

By Lemma \ref{lem:minimal} and Lemma \ref{lem:commute}, the invariance assumption {\bf [INV(H)]} implies that $z_{H}$ commutes with $h_{1}$.

For a left-ordering $<_{G}$ of $G$, an element $g \in G$ is called {\em $<_{G}$-cofinal} if for all $g' \in G$, there exist integers $m$ and $M$ such that $g^{m} <_{G} g' <_{G} g^{M}$ holds.
Although the cofinality assumptions {\bf [CF(G)]} and {\bf [CF(G)]} involve the generating sets $\mG$ and $\mH$, if we assume the invariance assumption  {\bf [INV(H)]} then these assumptions should be regarded as assumptions on $z_{G}$, $z_{H}$ and the isolated orderings $<_{G}$, $<_{H}$ as the next lemma shows. 

\begin{lemma}
\label{lem:cofinal}
Assume the invariance assumption {\bf [INV(H)]} is satisfied.
A generating set $\mH$ satisfying the cofinality assumption {\bf [CF(H)]} exists if and only if $z_{H}$ is $<_{H}$-positive cofinal and $H \neq \langle z_{H} \rangle$. Here $\langle z_{H} \rangle$ represents the subgroup of $H$ generated by $z_{H}$.
Moreover, in such case we may choose a generating set $\mH$ so that the cardinal of $\mH$ is equal to the rank of the isolated ordering $<_{H}$.
\end{lemma}
\begin{proof}
In the following, we assume the invariance assumption {\bf [INV(H)]}.

Assume that a generating set $\mH$ satisfies the cofinality assumption {\bf [CF(H)]}. Then by the invariance assumption {\bf [INV(H)]}, $z_{H}$ is $<_{H}$-positive cofinal and $H \neq \langle z_{H} \rangle$. 

We show the converse: if $z_{H}$ is $<_{H}$-positive cofinal and $H \neq \langle z_{H} \rangle$, then we can choose a generating set $\mH = \{h_{1},\ldots, h_{k}\}$ so that $\mH$ defines the isolated ordering $<_{H}$, and that $\mH$ satisfies {\bf [CF(H)]}. Moreover, we will show that we can choose $k$, the cardinal of $\mH$, so that $k$ is equal to $r(<_{H})$.

Let us take a generating set $\mH' =  \{h'_{1},\ldots, h'_{k} \}$ of $H$ which defines the isolated ordering $<_{H}$. By definition of rank, we may choose $\mH'$ so that $k=r(<_{H})$ holds.
With no loss of generality, we may assume that
\[ h'_{1} <_{H} \cdots <_{H} h'_{s} \leq_{H} z_{H} <_{H} h'_{s+1} <_{H} \cdots <_{H} h'_{k} .\]
Since $z_{H}$ is $<_{H}$-cofinal, for each $i$ there is a non-negative integer $N_{i}$ such that $1 <_{H} z_{H}^{-N_{i}}h'_{i} \leq z_{H}$. Let us put $h_{i}= z_{H}^{-N_{i}}h'_{i}$. By assumption, $h'_{i}=h_{i}$ if $i\leq s$.

By the hypothesis $H \neq \langle z_{H} \rangle$, we have a strict inequality $z_{H} >_{H} h'_{1}=h_{1}$.
Thus if necessary, by replacing $h_{i}$ with $h_{1}^{-1}h_{i}$, we may assume that $h_{i} \neq z_{H}$ for all $i$.

We show that $z_{H}$ is written as an $\{ h'_{1},\ldots, h'_{s}\}$-positive word. Assume that $z_{H} = V h'_{i} W$, where $i>s$ and $V,W$ are $\mH'$-positive or non-empty words. 
Then $z_{H} W^{-1}= Vh'_{i} >_{H} V z_{H}$, hence we get $1 \geq_{H} W^{-1} >_{H} z_{H}^{-1} Vz_{H}$. However, $<_{H}$ is a $z_{H}$-right invariant ordering, hence $z_{H}^{-1} V z_{H} \geq_{H} 1$. This is a contradiction.

Therefore, the generating set $\mH= \{ h_{1},\ldots, h_{k} \}$ also defines the isolated ordering $<_{H}$. By construction, $\mH$ is a generating set which satisfies the cofinality assumption {\bf [CF(H)]} with cardinal $k=r(<_{G})$.
\end{proof}

Thus, under the invariance assumption {\bf [INV(H)]}, we can always find a generating set $\mH$ which defines $<_{H}$ and satisfies the cofinality assumption {\bf [CF(H)]}, if the conditions on $<_{H}$ and $z_{H}$ in Lemma \ref{lem:cofinal} are satisfied. Moreover, if necessary we may choose $\mH$ so that the cardinal of $\mH$ is equal to the rank of $<_{H}$. 

Since for $z_{G}$ and $<_{G}$, the invariance assumption is automatically satisfied, we can always find a generating set $\mG$ which defines $<_{G}$ and satisfies the cofinality assumption {\bf [CF(G)]} if $z_{G}$ is $<_{G}$-positive cofinal and $G \neq \langle z_{G} \rangle$.

Now we put $\Delta_{H} = z_{H}h_{1}^{-1}$. Since $z_{H}$ and $h_{1}$ do not depend on the choice of the generating set $\mH$, the same holds for $\Delta_{H}$. As an element of $H$, $\Delta_{H}$ is characterized by the following property.

\begin{lemma}
\label{lem:DeltaH}
$\Delta_{H}$ is the $<_{H}$-maximal element which is strictly smaller than $z_{H}$.
\end{lemma}
\begin{proof}
Assume that $z_{H} h_{1}^{-1} = \Delta_{H} \leq_{H} h  <_{H} z_{H}$ holds for some $h \in H$. Then 
$h_{1}^{-1} \leq_{H} z_{H}^{-1}h <_{H} 1 $. By Lemma \ref{lem:minimal}, $h_{1}^{-1}$ is the $<_{H}$-maximal element which is strictly smaller than $1$, so $z_{H}^{-1}h = h_{1}^{-1}$.  Hence $h=z_{H}h_{1}^{-1}$. 
\end{proof}

Finally, we put $x_{i} = g_{i} \Delta_{H}^{-1} = g_{i}z_{H}^{-1}h_{1}$ and let $\mX = \{x_{1},\ldots, x_{m}\}.$ Then $\{\mX ,\mH\}$ generates the group $X$. The following lemma is rather obvious, but plays an important role in the proof of Theorem \ref{thm:main}.

\begin{lemma}
\label{lem:com2}
$z_{H}=z_{G}$ commutes with all $x_{i}$.
\end{lemma}
\begin{proof}
By Lemma \ref{lem:commute}, $z_{H}$ commutes with $\Delta_{H}=z_{H}h_{1}^{-1}$. Since $z_{H}=z_{G}$ commutes with all $g_{i}$, we conclude that $z_{H}$ commutes with all $x_{i} = g_{i}\Delta_{H}^{-1}$.
\end{proof}

\subsection{Property A and Property C criteria}

To prove that $\{\mX ,\mH\}$ defines an isolated left ordering $<_{X}$ of $X$, we use the following criterion which was used in the theory of the Dehornoy ordering of the braid groups \cite{ddrw} and Dehornoy-like orderings \cite{i1,n2}. Here we give the most general form of this kind of arguments.

\begin{definition}
Let $\mS=\{s_{1},\ldots,s_{m}\}$ be a generating set of a group $G$ and let 
 $W$ be a sub-semigroup of $(\mS \cup \mS^{-1})^{*}$.
\begin{enumerate}
\item We say $W$ has the {\em Property A (Acyclic Property)} if no word in $W$ represent the trivial element of $G$.
\item We say $W$ has the {\em Property C (Comparison Property)} if for each non-trivial element $g \in G$, either $g$ or $g^{-1}$ is represented by a word $w \in W$.
\end{enumerate}
\end{definition}

\begin{proposition}
\label{prop:defn}
Let $W$ be a sub-semigroup of $(\mS \cup \mS^{-1})^{*}$. 
Let $P= \pi(W)$, where $\pi: (\mS \cup \mS^{-1})^{*} \rightarrow G$ is the natural projection. 
Then $P$ is equal to a positive cone of a left ordering of $G$ if and only if $W$ has Properties A and C.
\end{proposition}
\begin{proof}
If $W$ is a positive cone of a left ordering, then it is obvious that $W$ has Properties A and C. 
We show the converse. Since $W$ is a sub-semigroup, $P$ is a sub-semigroup of $G$. By Property C, $G = P \cup \{1\} \cup P^{-1}$. 
Property A implies that $1 \not \in P$, hence $G$ is decomposed as a disjoint union $G=P \sqcup\{1\} \sqcup P^{-1}$. This shows that $P$ is a positive cone of a left ordering.
\end{proof}

\begin{definition}
The set of words $W$ in Proposition \ref{prop:defn} is called the language defining the corresponding left-ordering. 
\end{definition}

It is an interesting problem to ask whether or not one can choose a language defining an arbitrary left-ordering $<_{G}$ so that it is a regular language over finite alphabets: This is related to order-decision problems which we will consider in Section \ref{sec:comp}, but in this paper we will not treat this problem.

As a special case, we get a criterion for a finite generating set to define an isolated ordering, which will be used to show $\{\mX ,\mH\}$ indeed defines an isolated ordering.

\begin{corollary}
\label{cor:iso}
A finite generating set $\mG=\{g_{1},\ldots,g_{m}\}$ of a group $G$ defines an isolated ordering of $G$ if and only if the following conditions {\bf [Property A]} and {\bf [Property C]} hold:
\begin{description}
\item[Property A] If $g \in G$ is represented by a $\mG$-positive word, then $g \neq 1$.
\item[Property C] If $g \neq 1$, then $g$ is represented by either a $\mG$-positive or a $\mG$-negative word.
\end{description}
\end{corollary}

\subsection{Reduced standard factorization}
\label{sec:rsf}

Now we start to show that $\{\mX ,\mH\}$ indeed defines an isolated left ordering of $X$. From now on, we take $G,H,X,<_{G},<_{H},\mG,\mH,\mX$ as in assumptions in Theorem \ref{thm:main}, and we always assume the cofinality assumptions {\bf [CF(G)]}, {\bf [CF(H)]}, and the invariance assumption {\bf [INV(H)]}.

As the first step of the proof, we introduce a notion of reduced standard factorization, which serves as a certain kind of normal form of $X$ adapted to the generating set $\{\mX,\mH\}$.

Let $PX$ be the sub-semigroup of $X$ generated by $\mX=\{x_{1},\ldots,x_{m}\}$. 
A {\em standard factorization} of $x \in X$ is a factorization of $x \in X$ of the form
\[ \mF(x)= rp_{1}q_{1}\cdots p_{l}q_{l}\] 
where $r,q_{1},\ldots,q_{l} \in H$, $p_{1},\ldots,p_{l} \in PX$ satisfy the conditions
\begin{enumerate}
\item $q_{i} >_{H} 1$ $(i \neq l)$, and $q_{l} \geq_{H} 1$, and
\item $q_{i} \neq z_{H}^{N} $ for all $N>0$.
\end{enumerate}  

First we observe that every $x \in X$ admits a standard factorization.
Recall that $x_{i}^{-1} = \Delta_{H}g_{i}^{-1} =\Delta_{H}z_{H}^{-1}(z_{G}g_{i}^{-1})$. Since $1<_{G} (z_{G}g_{i}^{-1})$, $(z_{G}g_{i}^{-1})$ is written as a $\mG$-positive word. By rewriting such a $\mG$-positive word as an $\{\mX,\mH\}$-positive word, we may write $x_{i}^{-1}$ as
\[ x_{i}^{-1} = z_{H}^{-1}\Delta_{H}P_{i}(\mX,\mH) \]
where $P_{i}(\mX,\mH)$ denotes a certain $\{\mX,\mH\}$-positive word.
Thus, we can write $x$ without using $\{x_{1}^{-1},\ldots,x_{m}^{-1}\}$.

Moreover, $z_{H}$ is $<_{H}$-cofinal so for $h \in H$ there exists $N \in \Z$ such that $z_{H}^{N}h >_{H} 1$.
Recall that the isolated ordering $<_{H}$ is a $z_{H}$-right invariant ordering, and that $z_{H}$ commutes with all $\mX$ (Lemma \ref{lem:com2}). Thus, the above consideration shows that we are able to get a standard factorization
\[ \mF(x) = rp_{1}q_{1}\cdots p_{l}q_{l}.\] 

The {\em complexity} of a standard factorization $\mF(x)= rp_{1}q_{1}\cdots p_{l}q_{l}$ is defined to be $l$, and denoted by $c(\mF)$.

A {\em distinguished subfactorization} of a standard factorization $\mF(x)$ is, roughly saying, a part of the standard factorization $\mF(x)$ which can be regarded as a $\mG$-positive word, defined as follows.

We say a subfactorization 
\begin{equation}
\label{eqn:dsubfac}
 w= (q_{i} p_{i+1} q_{i+1} \cdots p_{i+r}q_{i+r})
\end{equation}
 in a standard factorization $\mF(x)$ is a {\em distinguished subfactorization}
if it satisfies the following two conditions: 
\begin{enumerate}
\item $q_{j} = \Delta_{H}$ for all $j = i,i+1,\ldots,i+r$.
\item $p_{j} \in \mX$ for all $j = i+1,\ldots,i+r$.
\end{enumerate}
That is, a distinguished subfactorization is a part of standard factorization which is written as 
\begin{equation}
\label{eqn:dist}
w = \Delta_{H} x_{j_{i+1}} \Delta_{H} x_{j_{i+2}} \cdots x_{j_{i+r}} \Delta_{H}
\end{equation}

We will express the distinguished subfactorization $w$ (\ref{eqn:dist}) by using a $\mG$-positive word $g_{w}$ as follows:
Let us take $x_{a} \in \mX$ so that $p'_{i} = p_{i} x_{a}^{-1} \in PX \cup \{1\}$ (such a choice of $x_{a}$ might be not unique), 
 and write a standard factorization $\mF(x)$ as
\begin{eqnarray*}
\mF(x) & =& rp_{1}q_{1}\cdots p_{l}q_{l}\\
 & = & rp_{1}q_{1} \cdots p_{i-1}q_{i-1}(p_{i}x_{a}^{-1})(x_{a}q_{i} \cdots p_{i+r}q_{i+r}) p_{i+r+1}q_{i+r+1}\cdots p_{l}q_{l}\\
 & = & rp_{1}q_{1} \cdots p_{i-1}q_{i-1}p'_{i} (x_{a}\Delta_{H}x_{j_{i+1}} \Delta_{H} x_{j_{i+2}} \cdots x_{j_{i+r}} \Delta_{H})p_{i+r+1}q_{i+r+1}\cdots p_{l}q_{l}.
\end{eqnarray*}

Let us put 
\[ g_{w} = g_{j_{i+1}}\cdots g_{j+i+r}. \]
We call $g_{w}$ the {\em corresponding $\mG$-positive word (element)} of the distinguished subfactorization $w$.
Since $g_{i}=x_{i} \Delta_{H}$, 
\[ x_{a}\Delta_{H}x_{j_{i+1}} \Delta_{H} x_{j_{i+2}} \cdots x_{j_{i+r}} \Delta_{H} = g_{a}g_{j_{i+1}}\cdots g_{j_{i+r}} = g_{a}g_{w}. \]
Thus, if $w$ is a distinguished subfactorization in $\mF(x)$, by choosing $x_{a}$ we may express $x$ as
\[ x = rp_{1}q_{1} \cdots p_{i-1}q_{i-1}p'_{i}[g_{a}g_{w}]p_{i+r+1}q_{i+r+1}\cdots p_{l}q_{l}. \]
by using the corresponding $\mG$-positive word $g_{w}$.

Next we introduce a notion of {\em reducible distinguished subfactorization}.
Let $w$ be a distinguished subfactorization of $\mF(x)$ as taken in (\ref{eqn:dsubfac}).
Let us take $x_{u} \in \mX$ so that $p'_{i+r+1}=x_{u}^{-1}p_{i+r+1} \in PX \cup \{1 \}$. As for the choice of $x_{a}$ above, such $x_{u}$ may not unique. If such $x_{u}$ does not exist, that is, $p_{i+r+1}=1$, we take $x_{u}=1$. We say a distinguished subfactorization $w$ is {\em reducible} if for {\em any} choice of such $x_{a}$ and $x_{u}$, we have an inequality $g_{a}g_{w}g_{u} \geq_{G} z_{G}$. Otherwise, that is, if one can choose $x_{a}$ and $x_{u}$ so that $g_{a}g_{w}g_{u} <_{G} z_{G}$ holds, then we say $w$ is {\em irreducible}.

Now we define the notion of a {\em reduced standard factorization}, which plays an important role in the proof of both Property A and Property C.

\begin{definition}[Reduced standard factorization]
Let $\mF(x)= rp_{1}q_{1}\cdots p_{l}q_{l}$ be a standard factorization.
We say $\mF$ is {\em reduced} if $q_{i} <_{H} z_{H}$ for all $i$ and $\mF$ contains no reducible distinguished subfactorization.
\end{definition}

We say a distinguished subfactorization $w$ of a standard factorization $\mF$ is {\em maximal} if there is no other distinguished subfactorization $w'$ of $\mF$ whose corresponding $\mG$-positive word $g_{w'}$ contains $g_{w}$ as its subword.
For any $<_{G}$-positive elements $g,g',g''$, since $z_{G}$ is central, if $g \geq_{G} z_{G}$ then $g'gg'' \geq_{G} z_{G}$. Thus, to see whether a standard factorization is reducible or not, it is sufficient to check that all maximal distinguished subfactorization are irreducible. 

\begin{example}
A distinguished subfactorization and related notions are slightly complex, so here we give an example.
Let us consider the case $\mX=\{x_{1},x_{2}\}$ and take a standard factorization of the form 
\begin{equation}
\label{eqn:exam}
\mF(x) = (x_{1}x_{2})\Delta_{H} x_{1} \Delta_{H} x_{2} \Delta_{H}(x_{1}^{3}x_{2})h_{1},
\end{equation}
for example.

In the standard factorization (\ref{eqn:exam})
$w = \Delta_{H}  x_{1} \Delta_{H}$ is a distinguished subfactorization.
The corresponding $\mG$-positive word is $g_{w}= g_{1}$.
In this case, we may choose $x_{a} = x_{2}$ since $(x_{1}x_{2})x_{2}^{-1} = x_{1} \in PX$. So we are able to write $x$ as
\[ x = x_{1}[g_{2}g_{1}] x_{2} \Delta_{H} (x_{1}^{3}x_{2})h_{1}. \]

However, the distinguished sub-factorization $w$ is not maximal: it is included in another distinguished subfactorization $w' =  \Delta_{H} x_{1} \Delta_{H} x_{2} \Delta_{H}$, and we may write
\[ x = x_{1}[g_{2}(g_{1}g_{2})](x_{1}^{3}x_{2})h_{1}. \]
The distinguished subfactorization $w'$ is maximal.

Is $w'$ reducible ? To see this, first we need to determine all possibilities of $x_{a}$ and $x_{u}$ in the definition of reducible distinguished subfactorization. 
Assume that $(x_{1}x_{2})x_{1}^{-1} \in PX \cup\{1\}$, but $x_{2}^{-1}(x_{1}^{3}x_{2}) \not \in PX \cup\{1\}$. Then $x_{a}=x_{1}$ or $x_{2}$, and $x_{u}=x_{1}$.
Hence our definition says, $w'$ is reducible if and only if 
\[ g_{1}(g_{1}g_{2})g_{1}\geq_{G} z_{G},  \textrm{ and } g_{2}(g_{1}g_{2})g_{1}\geq_{G} z_{G}\]
hold.
\end{example}

First we show the existence of the reduced standard subfactorization.
The proof of the next lemma utilizes the standard form of amalgamated free products, and mainly works in the generating set $\{\mG,\mH\}$. 

\begin{lemma}
\label{lem:stand}
Every element $x \in X$ admits a reduced standard subfactorization.
\end{lemma}
\begin{proof}
Since $X$ is an amalgamated free product of $G$ and $H$, every $x \in X$ is written as
\[ x= q_{0}f_{1}q_{1}f_{2}q_{2} \cdots f_{l} q_{l}\]
where $ q_{i} \in H$, $f_{i} \in G$, and $q_{i} \neq z_{H}^{N}$ and $f_{i} \neq z_{G}^{N}$ for any $N \in \Z$ and $i>0$.

Since $z_{G}$ is $<_{G}$-cofinal, for each $i>0$ there exists $N_{i} \in \Z$ which satisfies 
\[  z_{G}^{N_{i}} <_{G} f_{i} <_{G} z_{G}^{N_{i}+1}.\]
We put $f_{i}^{*}=z_{G}^{-N_{i}}f_{i}$. Then $f_{i}^{*}$
 satisfies the inequality
\[ 1 <_{G} f_{i}^{*} <_{G} z_{G}\]
Similarly, since $z_{H}$ is $<_{H}$-cofinal, for each $i>0$ there exists $M_{i}$ which satisfies the inequality
\[ z_{H}^{M_{i}} \leq_{H} \Delta_{H}q_{i}<_{H} z_{H}^{M_{i}+1}.\]

Let $L_{i} = \sum_{j>i} (N_{j}+M_{j})$, and put $q_{i}^{*}=z_{H}^{-L_{i}}(z_{H}^{-M_{i}}\Delta_{H}q_{i})z_{H}^{L_{i}}$.
Since $<_{H}$ is a $z_{H}$-right invariant ordering, $1 \leq_{H} q_{i}^{*} <_{H} z_{H}$ holds. We have assumed that $q_{i} \neq z_{H}^{N}$, so we have $q_{i}^{*} \neq \Delta_{H}$. Thus, $1\leq_{H} q^{*}_{i} <_{H} \Delta_{H}$.

Then we get a reduced standard factorization of $x$ as follows.
First we modify the first expression of $x$ as
\begin{eqnarray*}
x &=& q_{0}f_{1}q_{1} \cdots f_{l}q_{l} \\
& = & q_{0}(z_{G}^{N_{1}}f^{*}_{1}) q_{1}(z_{G}^{N_{2}}f^{*}_{2}) \cdots (z_{G}^{N_{l-1}}f^{*}_{l-1})q_{l-1}(z_{G}^{N_{l}} f^{*}_{l})q_{l}\\
& = & (q_{0}z_{H}^{N_{1}})f^{*}_{1} (q_{1}z_{H}^{N_{2}}) \cdots f^{*}_{l-1}(q_{l-1}z_{H}^{N_{l}}) f^{*}_{l}q_{l}\\
& = & (q_{0}z_{H}^{N_{1}})f^{*}_{1} (q_{1}z_{H}^{N_{2}}) \cdots 
f^{*}_{l-1}(q_{l-1}z_{H}^{N_{l}}) f^{*}_{l}\Delta_{H}^{-1}z_{H}^{M_{l}}(z_{H}^{-M_{l}}\Delta_{H}q_{l})\\
& = & (q_{0}z_{H}^{N_{1}})f^{*}_{1} (q_{1}z_{H}^{N_{2}}) \cdots 
f^{*}_{l-1}(q_{l-1}z_{H}^{N_{l}}z_{H}^{M_{l}})(f_{l}^{*}\Delta_{H}^{-1})q_{l}^{*}\\
& = & (q_{0}z_{H}^{N_{1}})f^{*}_{1} (q_{1}z_{H}^{N_{2}}) \cdots 
f^{*}_{l-1}\Delta_{H}^{-1} z_{H}^{N_{l}+M_{l}}( z_{H}^{-N_{l}-M_{l}}  \Delta_{H}q_{l-1}z_{H}^{N_{l}+M_{l}}) (f^{*}_{l}\Delta_{H}^{-1})q_{l}^{*}\\
& = & (q_{0}z_{H}^{N_{1}})f^{*}_{1} (q_{1}z_{H}^{N_{2}}) \cdots   z_{H}^{N_{l}+M_{l}}(f_{l-1}^{*}\Delta_{H}^{-1})q_{l-1}^{*} (f^{*}_{l}\Delta_{H}^{-1})q_{l}^{*}\\
& = & \cdots\\
& = & (q_{0}z_{H}^{L_{0}}) (f_{1}^{*}\Delta_{H}^{-1})q_{1}^{*} \cdots (f_{l-1}^{*}\Delta_{H}^{-1})q_{l-1}^{*} (f^{*}_{l}\Delta_{H}^{-1})q_{l}^{*}.
\end{eqnarray*}

Now let us write $f^{*}_{i} = P_{i}(\mG)g_{k_{i}}$, where $P_{i}(\mG)$ is a $\mG$-positive or an empty word. 
Since $g_{i}=x_{i}\Delta_{H}$, we may express $P_{i}(\mG)$ as
 an $\{\mX,\mH\}$-positive (or, empty) word. Hence by rewriting each $P_{i}(\mG)$ as
 an $\{\mX,\mH\}$-positive (or, empty) word, we get a standard factorization 
\[ \mF(x)= (q_{0}z_{H}^{L_{0}}) [P_{1}(\mG)] x_{k_{1}} q_{1}^{*} \cdots [P_{l-1}(\mG)] x_{k_{l-1}} q_{l-1}^{*} [P_{l}(\mG)] x_{k_{l}} q_{l}^{*}. \]

Recall that $q_{i}^{*} \neq \Delta_{H}$ for all $i$. Thus for a maximal distinguished subfactorization $w$ in $\mF(x)$, we may choose $g_{a}$ and $g_{u}$ so that $g_{a}g_{w}g_{u} = P_{i}(\mG)g_{k_{i}}$ holds for some $i$.
Since $P_{i}(\mG)g_{k_{i}} = f_{i}^{*} <_{G} z_{G}$, this implies that all distinguished sub-factorizations are irreducible. 
Hence $\mF(x)$ is a reduced standard factorization.
\end{proof}

\subsection{Reducing operation and the proof of Property A}

In the proof of Lemma \ref{lem:stand} given in previous section, we mainly used the generating set $\{\mG,\mH\}$. In this section we give an alternative way to get a reduced standard factorization, which works mainly in the generating set $\{\mX,\mH\}$. 

We say a standard factorization $\mF(x) = rp_{1}q_{1}\cdots p_{l}q_{l}$ is {\em pre-reduced} if $1<_{H} q_{i} <_{H} z_{H}$ holds for all $i$.
It is rather easy to see pre-reduced standard factorization exists.

\begin{lemma}[Existence of pre-reduced standard factorization]
\label{lem:pre-reduced}

Every element $x \in X$ admits a pre-reduced standard factorization.
\end{lemma}
\begin{proof}
Let $ \mF(x) = rp_{1}q_{1}\cdots p_{l}q_{l}$ 
be a standard factorization.
For each $i$, take $M_{i}\geq 0$ so that $z_{H}^{M_{i}} <_{H} q_{i} <_{H} z_{H}^{M_{i}+1}$.
Let  $L_{i} = \sum_{j\geq i} M_{i}$ and $q_{i}^{*} = z_{H}^{-L_{i}}q_{i}z_{H}^{L_{i+1}} =  z_{H}^{-L_{i+1}}(z_{H}^{-M_{i}}q_{i})z_{H}^{L_{i+1}}$. Since $<_{H}$ is $z_{H}$-right invariant, $1<_{H} q^{*}_{i} <_{H} z_{H}$. Therefore, we get a pre-reduced standard factorization 
\begin{eqnarray*}
x & = & rp_{1}q_{1}\cdots p_{l}q_{l}\\
  & = & rp_{1}q_{1}\cdots p_{l-1}q_{l-1}p_{l}(z_{H}^{M_{l}}q^{*}_{l}) \\
  & = & rp_{1}q_{1}\cdots p_{l-1}(q_{l-1}z_{H}^{M_{l}})p_{l}q^{*}_{l}\\
  & = & rp_{1}q_{1}\cdots p_{l-1}z_{H}^{-M_{l}-M_{l-1}}(z_{H}^{-M_{l}-M_{l-1}}q_{l-1}z_{H}^{M_{l}})p_{l}q^{*}_{l}\\
  & = & rp_{1}q_{1}\cdots p_{l-1}z_{H}^{-L_{l-1}}q^{*}_{l-1}p_{l}q^{*}_{l}\\
  & = & \cdots \\
  & = & (rz_{H}^{-L_{0}})p_{1}q^{*}_{1} \cdots p_{l}q^{*}_{l}.
  \end{eqnarray*}
\end{proof}

To show we are actually able to get reduced standard factorization, we observe that we are able to eliminate all reducible distinguished subfactorizations. Let $d(\mF)$ be the number of maximal reducible distinguished subfactorizations. The next lemma gives alternative proof that a reduced standard factorization exists. It says that by induction on $(d(\mF),c(\mF))$ for pre-reduced factorization $\mF$, we are able to get reduced standard factorization. 

\begin{lemma}[Reducing operation]
\label{lem:reducing}
Let $\mF(x) = rp_{1}q_{1}\cdots p_{l}q_{l}$ be a pre-reduced standard factorization of $x \in X$. 
If $\mF(x)$ contains a reducible distinguished subfactorization, then we can find another pre-reduced standard factorization $\mF'(x) = r'p'_{1}q'_{1}\cdots $ which satisfies $d(\mF') < d(\mF)$ or, $d(\mF')=d(\mF)$ and $c(\mF') < c(\mF)$.
Moreover, if $r>_{H} 1$ then $r'>_{H}1$.
\end{lemma}
\begin{proof}
Let $w=q_{i}p_{i+1}\cdots p_{s-1}q_{s-1}$ be a reducible maximal distinguished subfactorization in $\mF(x)$.
Thus, we may assume that the pre-reduced standard factorization $\mF(x)$ is written as
\[ \mF(x) =  r p_{1} q_{1} \cdots p_{i-1}q_{i-1}p'_{i}[g_{a}g_{w}]x_{u}p'_{s}q_{s} \cdots  p_{l}q_{l}\]
where 
\begin{enumerate}
\item $p'_{i}=p_{i}x_{a}^{-1}$ and $p'_{s}=x_{u}^{-1}p_{s}$
\item $p'_{i},p'_{s} \in PX \cup\{1\}$.
\item $g_{a}g_{w}g_{u} \geq_{G} z_{G}$. 
\end{enumerate} 

Now take $N>0$ so that $z_{G}^{N} <_{G} g_{a}g_{w}g_{u} \leq_{G} z_{G}^{N+1}$, and for $j<i$ let $q^{*}_{j} = z_{H}^{-N}q_{j} z_{H}^{N}$.
Then we may write $x$ as
\begin{eqnarray*}
x & = & r p_{1} q_{1} \cdots p_{i-1}q_{i-1}p'_{i}[g_{a}g_{w}]x_{u}p'_{s}q_{s} \cdots  p_{l}q_{l}\\
 & = & r p_{1} q_{1} \cdots p_{i-1}q_{i-1}p'_{i} z_{G}^{N}(z_{G}^{-N}g_{a}g_{w}g_{u})\Delta_{H}^{-1} p'_{s}q_{s} \cdots  p_{l}q_{l}\\
 & = & (rz_{H}^{N}) p_{1} q^{*}_{1} \cdots p_{i-1}q^{*}_{i-1} p'_{i} (z_{G}^{-N}g_{a}g_{w}g_{u}) \Delta_{H}^{-1} p'_{s}q_{s} \cdots  p_{l}q_{l}.
 \end{eqnarray*}
 
 First assume that  $(z_{G}^{-N}g_{a}g_{w}g_{u})=z_{G}=z_{H}$. Then we write $x$ as
 \begin{eqnarray*}
x & = &(rz_{H}^{N}) p_{1} q^{*}_{1} \cdots p_{i-1}q^{*}_{i-1} p'_{i}(z_{H}\Delta_{H}^{-1}) p'_{s} q_{s} \cdots  p_{l}q_{l}\\
& = & (rz_{H}^{N}) p_{1} q^{*}_{1} \cdots p_{i-1}q^{*}_{i-1}p'_{i}h_{1}p'_{s}q_{s} \cdots  p_{l}q_{l}
\end{eqnarray*}

If $p'_{i} \neq 1$ and $p'_{s} \neq 1$, then we get a pre-reduced standard factorization 
\begin{equation}
\label{eqn:red1}
 \mF'(x) = (rz_{H}^{N}) p_{1} q^{*}_{1} \cdots p_{i-1}q^{*}_{i-1}p'_{i} h_{1} p'_{s}q_{s} \cdots  p_{l}q_{l}. 
\end{equation}
In $\mF'(x)$, we removed the reducible distinguished subfactorization $w$ and no distinguished subfactorization is created, so $d(\mF') < d(\mF)$.

If $p'_{i}=1$ or $p'_{s} = 1$, then the standard factorization (\ref{eqn:red1}) might fail to be pre-reduced. We construct a pre-reduced standard factorization $\mF''$ from the standard factorization (\ref{eqn:red1}) by using the argument of proof of Lemma \ref{lem:pre-reduced}. In such case, we might produce one new reducible maximal distinguished subfactorization, so in general $d(\mF'') \leq d(\mF)$ although we have removed $w$ from $\mF(x)$. 
In this case we have $c(\mF'') < c(\mF)$.

(Here is a simple example where $d(\mF'')$ does not decrease: assume that $p'_{i}=1$, $p'_{s} \in \mX$, and $q^{*}_{i-1}h_{1} = \Delta_{H}$, and that $w=q_{s}\cdots$ is a distinguished subfactorization: then we get a new maximal distinguished subfactorization $w = \cdots (q^{*}_{i-1}h_{1}) p'_s q_s \cdots $ in $\mF'(x)$. This maximal distinguished subfactorization might be reducible, so $d(\mF'')=d(\mF)$ may occur.)

Next assume that $(z_{G}^{-N}g_{a}g_{w}g_{u}) \neq z_{G}$. Let us put $g' =  (z_{G}^{-N}g_{a}g_{w}g_{u})g_{1}^{-1}$ and write $x$ as
\begin{eqnarray*}
\mF(x) &=& (rz_{H}^{N}) p_{1} q^{*}_{1} \cdots p_{i-1}q^{*}_{i-1} p'_{i}(z_{G}^{-N}g_{a}g_{w}g_{u})\Delta_{H}^{-1} p'_{s}q_{s} \cdots  p_{l}q_{l}\\
& = & (rz_{H}^{N}) p_{1} q^{*}_{1} \cdots p_{i-1}q^{*}_{i-1}p'_{i} g'g_{1}\Delta_{H}^{-1} p'_{s}q_{s} \cdots  p_{l}q_{l}\\
  & = & (rz_{H}^{N}) p_{1} q^{*}_{1} \cdots p_{i-1}q^{*}_{i-1}p'_{i}g' (x_{1}p'_{s})q_{s} \cdots  p_{l}q_{l}.
 \end{eqnarray*} 

If $g'=1$, then we get a pre-reduced standard factorization
\[ \mF'(x) = (rz_{H}^{N}) p_{1} q^{*}_{1} \cdots p_{i-1}q^{*}_{i-1}(p'_{i} x_{1}p'_{s})q_{s} \cdots  p_{l}q_{l}. \]
such that $d(\mF') < d(\mF)$.

If $g'>_{G}1$, then let us write $g' = g_{a'}P(\mG)$ where $P(\mG)$ is a $\mG$-positive, or empty word. By rewriting $P(\mG)$ as an $\{\mX,\mH\}$-positive word, we get a new pre-reduced standard factorization 
\[ \mF'(x) = (rz_{H}^{N}) p_{1} q^{*}_{1} \cdots p_{i-1}q^{*}_{i-1}p'_{i} [g_{a'}P(\mG)] (x_{1}p'_{s})q_{s} \cdots  p_{l}q_{l}. \]

Observe that $P(\mG)$ gives rise to a maximal distinguished subfactorization $w'$ in $\mF'(x)$ such that $g_{w'}=P(\mG)$. By Lemma \ref{lem:minimal}, $<_{G}$ is a $g_{1}$-right invariant ordering, so $1 \leq_{G} g' <_{G} z_{G}g_{1}^{-1}$, so $g_{a'} g_{w'} g_{1} <_{G} z_{G}$. Hence the maximal distinguished subfactorization $w'$ in $\mF'(x)$ is irreducible. 
By construction, all other maximal reducible distinguished subfactorizations in $\mF'(x)$ are derived from the pre-reduced factorization $\mF(x)$. Since we have removed the maximal reducible distinguished subfactorization $w$ in $\mF(x)$, so $d(\mF') < d(\mF)$.

Moreover, by construction we have always $r \leq_{H} r'$. In particular, $1<_{H}r'$ if $1<_{H}r$,
\end{proof}

Now we are ready to prove Property A.

\begin{proposition}[Property A]
\label{prop:PropertyA}
If $x$ is expressed as an $\{\mX, \mH\}$-positive word, then $x \neq 1$.
\end{proposition}
\begin{proof}
Assume that $x$ is expressed by an $\{\mX,\mH\}$-positive word. Such a word expression can be modified to a standard factorization which is also an $\{\mX,\mH\}$-positive word: By the proof of Lemma \ref{lem:pre-reduced}, we can modify such a standard factorization so that it is pre-reduced, preserving the property that it is also an $\{\mX,\mH\}$-positive word. 
By Lemma \ref{lem:reducing}, we may modify the $\{\mX,\mH\}$-positive pre-reduced standard expression $\mF(x)$ so that it is an $\{\mX,\mH\}$-positive reduced standard factorization.

Now let us rewrite $\mF(x)$ as a word on $\{\mG,\mH\}$ as follows.
Let $w$ be a maximal distinguished subfactorization in $\mF(x)$ so we may write $\mF(x)$ as
\[ \mF(x)= r p_{1} q_{1} \cdots p_{i-1}q_{i-1}p'_{i}[g_{a}g_{w}]p_{s}q_{s} \cdots  p_{l}q_{l} \]
Since $w$ is irreducible, we may choose $g_{a}$ and $x_{u} \in \mX$ so that $p'_{s} = x_{u}^{-1} p_{s}, p'_{i} \in PX \cup\{1\}$ and that $g_{a}g_{w}g_{u} <_{G} z_{G}$.
Then we write $x$ as
\[ \mF(x) = r p_{1} q_{1} \cdots p_{i-1}q_{i-1}p'_{i} (g_{a}g_{w}g_{u}) \Delta_{H}^{-1} p'_{s}q_{s} \cdots  p_{l}q_{l}, \]
and regard $(g_{a}g_{w}g_{u})$ as a $\mG$-positive word.

Iterating this rewriting procedure for each maximal distinguished subword, and rewriting the rest of $x_{i}$ in $\mF(x)$ as a word on $\{\mG,\mH\}$ by using the relation $x_{i}=g_{i}\Delta_{H}^{-1}$, we finally write $x$ as
\[ x = W_{0}V_{1}W_{1}\cdots V_{n}W_{n}\] 
where $W_{i}$ is a word on $\mH^{\pm 1}$ and $V_{i}$ is a word on $\mG^{\pm 1}$. 
By construction, $V_{i} \in \mX$ or $V_{i} = g_{a}g_{w}g_{u}$ where $g_{w}$ is a maximal distinguished subfactorization in $\mF(x)$. Since we have chosen $g_{a}g_{w}g_{u} <_{G} z_{G}$, this implies that, $V_{i} \not \in \langle z_{G} \rangle$ for all $i$. Similarly, the assumption that $\mF(x)$ is reduced implies that we may choose $W_{i} \not \in  \langle z_{H} \rangle$ for $i>0$. This implies that $x \neq 1$, since $X= G*_{\langle z_{G}=z_{H} \rangle} H$.

\end{proof}

\subsection{Proof of Property C}

Next we give a proof of Property C.
To begin with, we observe a simple, but useful observation.

\begin{lemma}
\label{lem:basic}
\[ h_{j}^{-1}x_{i} = N(\mX,\mH) \Delta_{H}^{-1} \]
where $N(\mX,\mH)$ represents an $\{\mX,\mH\}$-negative word.
\end{lemma}
\begin{proof}
Since $z_{H}=z_{G}$ and $x_{i}=g_{i}\Delta_{H}^{-1}$, we have
\[ z_{H}=  g_{i}g_{i}^{-1}z_{G} g_{1}^{-1}g_{1}=x_{i}\Delta_{H}( g_{i}^{-1} z_{G} g_{1}^{-1}) x_{1} \Delta_{H}.\]
Therefore 
\[ h_{j}^{-1}x_{i} =(h_{j}^{-1} z_{H}\Delta_{H}^{-1}) x_{1}^{-1}(z_{G}^{-1}g_{1}g_{i}) \Delta_{H}^{-1} = (h_{j}^{-1}h_{1})x_{1}^{-1}(z_{G}^{-1}g_{1}g_{i})\Delta_{H}^{-1}.\]
Since $z_{G}^{-1}g_{i} <_{G} 1$ and $g_{1}$ is the $<_{G}$-minimal positive element, $z_{G}^{-1}g_{i} \leq_{G} g_{1}^{-1}$. Hence $z_{G}^{-1}g_{1}g_{i} \leq_{G} 1$. Thus, $(h_{j}^{-1}h_{1})x_{1}^{-1}(z_{G}^{-1}g_{1}g_{i})$ is written as an $\{\mX , \mH\}$-negative word.

\end{proof}
Now we are ready to prove Property $C$.

\begin{proposition}[Property C]
\label{prop:PropertyC}
Each non-trivial element $x \in X$ is expressed by an $\{\mX,\mH\}$-positive word or an $\{\mX,\mH\}$-negative word.
\end{proposition}
\begin{proof}
Let $x$ be a non-trivial element of $X$ and take a reduced standard factorization of $x$,
\[ \mF(x)= rp_{1}q_{1}\cdots p_{l}q_{l}. \]
If $r \geq_{H} 1$, $r$ can be written as an $\mH$-positive or empty word, hence we may express $x$ as an $\{\mX,\mH\}$-positive word.

By induction on $l=c(\mF)$, we prove that $x$ is expressed by an $\{\mX,\mH\}$-negative word under the assumption that $r<_{H}1$.

First assume that $q_{1} \neq \Delta_{H}$.
Since $r <_{H} 1$, we can express $r$ as $r = N(\mH) h_{1}^{-1}$, where $N(\mH)$ is an $\mH$-negative word or an empty word.
Take an $\mX$-positive word expression of $p_{1} = x_{i_{1}}x_{i_{2}}\cdots x_{i_{p}}$.
Then by Lemma \ref{lem:basic},
\begin{eqnarray*}
rp_{1} p_{2}q_{2} \cdots & = &  (N(\mH) h_{1}^{-1}) (x_{i_{1}}x_{i_{2}}\cdots x_{i_{p}})  q_{1} p_{2}q_{2} \cdots\\
& = & N(\mH) (h_{1}^{-1} x_{i_{1}})x_{i_{2}}\cdots x_{i_{p}} q_{1} p_{2}q_{2} \cdots \\
 		& = & N(\mX,\mH)\Delta_{H}^{-1}x_{i_{2}}\cdots x_{i_{p}} q_{1} p_{2}q_{2} \cdots \\
 		& = & N(\mX,\mH) (h_{1}^{-1} x_{i_{2}})\cdots x_{i_{p}} q_{1} p_{2}q_{2} \cdots \\
		& =  & \cdots \\
		& = & N(\mX,\mH) \Delta_{H}^{-1} q_{1} p_{2}q_{2} \cdots.
\end{eqnarray*}
 $N(\mX,\mH)$ represents an $\{\mX,\mH\}$-negative word.
 
Since $\mF(x)$ is a reduced standard factorization, $q_{1} <_{H} z_{H}$. By Lemma \ref{lem:DeltaH} $\Delta_{H}$ is the $<_{H}$-maximal element of $H$ which is strictly smaller than $z_{H}$, so $q_{1} \leq_{H} \Delta_{H}$. We have assumed that $q_{1} \neq \Delta_{H}$ so $(\Delta_{H}^{-1}q_{1}) <_{H} 1$. Thus the subword $ (\Delta_{H}^{-1}q_{1})  p_{2}q_{2} \cdots  p_{l}q_{l}$ is a reduced standard factorization with complexity $(l-1)$. By induction, $ (\Delta_{H}^{-1}q_{1})  p_{2}q_{2} \cdots  p_{l}q_{l}$ is written as an $\{\mX,\mH\}$-negative word, hence we conclude that $x$ is written as an $\{\mX,\mH\}$-negative word.

Next assume that $q_{1}=\Delta_{H}$. Let $w= q_{1}p_{2}q_{2}\cdots p_{s-1}q_{s-1}$ be a maximal distinguished subfactorization of $\mF(x)$ which contains $q_{1}$. Thus, the reduced standard factorization $S$ is written as
\[ \mF(x) = r p'_{1} [g_{a} g_{w}] x_{u} p'_{s}q_{s}p_{s+1} \cdots p_{l}q_{l} \]
where $p'_{1}=p_{1}x_{a}^{-1}, p'_{s} = x_{u}^{-1}p_{s} \in PX \cup \{ 1 \}$.

Then by Lemma \ref{lem:basic},
\begin{eqnarray*}
x & = &   r p'_{1} [g_{a}g_{w}] x_{u} p'_{s}q_{s}p_{s+1} \cdots p_{l}q_{l} \\
 & = & N(\mX,\mH) h_{1}^{-1}[g_{a}g_{w}]x_{u}\Delta_{H} \Delta_{H}^{-1}p'_{s}q_{s} \cdots p_{l}q_{l}\\
 & = & N(\mX,\mH) h_{1}^{-1}[g_{a}g_{w}g_{u}] \Delta_{H}^{-1}p'_{s}q_{s} \cdots p_{l}q_{l}\\
 & = & N(\mX,\mH) \Delta_{H}(z_{G}^{-1}g_{a}g_{w}g_{u}) \Delta_{H}^{-1}p'_{s}q_{s} \cdots p_{l}q_{l}\\
\end{eqnarray*}

The distinguished subfactorization $w$ is irreducible so we may choose $x_{u}$ and $g_{a}$ so that $z_{G}^{-1}g_{a}g_{w}g_{u} <_{G} 1$ holds.
This implies that $z_{G}^{-1}g_{a}g_{w}g_{u}$ is written as a $\mG$-negative word. By expressing a $\mG$-negative word expression of $z_{G}^{-1}g_{a}g_{w}g_{u}$ as an $\{\mX,\mH\}$-negative word, we conclude that $z_{G}^{-1}g_{a}g_{w}g_{u}$ is written as a word of the form $\Delta_{H}^{-1} N(\mX,\mH)$. Hence
\begin{eqnarray*} 
x & = & N(\mX,\mH) \Delta_{H} \Delta_{H}^{-1} N(\mX,\mH) (\Delta_{H}^{-1}p'_{s} q_{s} \cdots p_{l}q_{l})\\
 & = & N(\mX,\mH) (\Delta_{H}^{-1}p'_{s} q_{s} \cdots p_{l}q_{l}).
\end{eqnarray*}

If $p'_{s}\neq 1$, then $(\Delta_{H}^{-1}p'_{s}q_{s}\cdots p_{l}q_{l})$ is a reduced standard factorization having the complexity less than $l$. Hence by induction, $(\Delta_{H}^{-1}p'_{s}q_{s}\cdots p_{l}q_{l})$ is expressed by an $\{\mX, \mH\}$-negative word.

If $p'_{s}=1$, then $q_{s} \neq \Delta_{H}$ since $w$ was a maximal distinguished subfactorization. Hence $q_{s} <_{H} \Delta_{H}$, and $(\Delta_{H}^{-1}q_{s})p_{s+1} \cdots p_{l}q_{l}$ is a reduced standard factorization with complexity less than $l$. By induction, $(\Delta_{H}^{-1}q_{s})p_{s+1} \cdots p_{l}q_{l}$ is expressed by an $\{\mX, \mH\}$-negative word. 

Thus in either case, we conclude $x$ is expressed by an $\{\mX,\mH\}$-negative word.
\end{proof}

\subsection{Proof of Theorem \ref{thm:main}}
Now we are ready to prove our main theorem.

\begin{proof}[of Theorem \ref{thm:main}]

(i): In Proposition \ref{prop:PropertyA} and Proposition \ref{prop:PropertyC}, we have already confirmed the Properties $A$ and $C$ for the generating set $\{\mX ,\mH\}$. By Corollary \ref{cor:iso} the generating set $\{\mX ,\mH\}$ indeed defines an isolated left ordering $<_{X}$ of $X$.\\

(ii):
Let $\mG' = \{g'_{1}, \ldots \}$ and $\mH' = \{h'_{1}, \ldots \}$ be other generating sets of $G$ and $H$ satisfying {\bf[CF(G)]} and {\bf [CF(H)]}.
Recall that $\Delta_{H} = z_{H}h_{1}^{-1}$ does not depend on the choice of a generating set $\mH$. 
Let $x_{i}= g_{i} \Delta_{H}^{-1}$, $x'_{i} = g'_{i}\Delta_{H}^{-1}$, $\mX = \{x_{1},\ldots,\}$, and  $\mX' =\{x'_{1},\ldots,\}$.

Since $\mH$ and $\mH'$ are generators of the same semigroup, we may write $h_{i}$ as an $\mH'$-positive word. Similarly, since $\mG$ and $\mG'$ are generators of the same semigroup, we may write $g_{i}$ as a $\mG'$-positive word $g_{i} = g'_{i_{1}} g'_{i_{2}} \cdots g'_{i_{l}}$.
Thus,
\[ x_{i} =  g_{i} \Delta_{H}^{-1} = g'_{i_{1}} g'_{i_{2}} \cdots g'_{i_{l}}\Delta_{H}^{-1} = x'_{i_{1}} \Delta_{H} x'_{i_{2}} \Delta_{H} \cdots x'_{i_{l-1}}\Delta_{H} x'_{i_{l}}\]
so $x_{i}$ is written as an $\{\mX',\mH'\}$-positive word.
Thus, if $x \in X$ is expressed by an $\{\mX ,\mH\}$-positive word, then $x$ is also represented by an $\{\mX',\mH'\}$-positive word. By interchanging the roles of $\{\mG, \mH\}$ and $\{ \mG',\mH'\}$, we conclude that $\{\mX ,\mH\}$ and $\{\mX',\mH'\}$ generates the same sub-semigroup of $X$ so they define the same isolated ordering of $X$. \\

(iii): This is obvious from the definition of $<_{X}$. \\

(iv): The inequality $h_{1}<_{X} h_{2} <_{X}\cdots <_{X} h_{n}$ follows from the definition of $<_{X}$. By Lemma \ref{lem:basic}, $x_{i} <_{X} h_{1}$ for all $i$. Now we show $x_{i}<_{X} x_{j}$ if $i<j$. Since $g_{i}<_{G}g_{j}$ if $i<j$, $g_{i}^{-1}g_{j}$ is written as a $\mG$-positive word. Now by definition $g_{i} = x_{i}\Delta_{H}$, so we may express a $\mG$-positive word expression of $g_{i}^{-1}g_{j}$ as an $\{\mX,\mH\}$-positive word expression of the form $P_{i,j}(\mX,\mH) \Delta_{H}$, where $P_{i,j}(\mX,\mH)$ represents an $\{\mX,\mH\}$-positive word. Therefore
 $x_{i}^{-1}x_{j}= \Delta_{H}g_{i}^{-1}g_{j}\Delta_{H}^{-1} = \Delta_{H} P_{i,j}(\mX,\mH)$, so $x_{i} <_{X} x_{j}$. The assertion that $z=z_{G}=z_{H}$ is $<_{X}$-positive cofinal is obvious. 
To see that $<_{X}$ is a $z$-right invariant ordering, we observe that $z^{-1}x_{i}z = x_{i} >_{X} 1$ and $z^{-1} h_{j} z >_{X} 1$. 
Now for $x,x' \in X$, assume $x <_{X} x'$, so $x^{-1}x'$ is written as $\{\mX,\mH\}$-positive word $w=s_{1}\cdots s_{m}$, where $s_{i}$ denotes $x_{j}$ or $h_{j}$.
Then $z^{-1} (x^{-1} x') z = (z^{-1} s_{1} z) \cdots (z^{-1} s_{m} z) >_{X} 1$, hence $xz <_{X} x'z$.\\

(v): Recall that by Lemma \ref{lem:cofinal}, we may choose the generating sets $\mG$ and $\mH$ so that the cardinal of $\mG$, $\mH$ are equal to $r(<_{G})$, $r(<_{H})$ respectively. Thus, $r(<_{X}) \leq r(<_{G}) +r(<_{H})$. \\

(vi): We prove that $\langle x_{1} \rangle$ is the unique $<_{X}$-convex non-trivial proper subgroup of $X$. Recall by (2), (4) and Lemma \ref{lem:minimal}, $x_{1}$ is the minimal $<_{X}$-positive element of $X$, hence $x_{1}$ does not depend on a choice of $\mG$ and $\mH$. In particular, $\langle x_{1} \rangle$ is a non-trivial $<_{X}$-convex subgroup.

Let $C$ be a $<_{X}$-convex subgroup of $X$.
Assume that $C \supset \langle x_{1} \rangle $.
Let $y \in C - \langle x_{1} \rangle $ be an $<_{X}$-positive element.
Then $y$ is written as  $y = x_{1}^{m} x_{j} P(\mX,\mH)$ or $y=x_{1}^{m} h_{l}P(\mX,\mH)$ where $m\geq 0$, $l>0$, $j>1$ and $P(\mX,\mH)$ is an $\{\mX,\mH\}$-positive word. Since $x_{1} \in C$, we may choose $y$ so that $m=0$ by considering $x_{1}^{-m}y$ instead.

First we consider the case $\mX \not \subset \langle x_{1} \rangle$. Then we may choose $y$ so that $1< x_{2} \leq_{X} y$ holds, so the convexity assumption implies $x_{2} \in C$.
Now observe that $x_{1}^{-1}x_{2} = \Delta_{H}g_{1}^{-1}g_{2}\Delta_{H}^{-1} = \Delta_{H} P_{1,2}(\mX,\mH)$, hence
\[ 1<_{X} h_{p} \leq z_{H}h_{1}^{-1} =\Delta_{H} <_{X}\Delta_{H} P_{1,2}(\mX,\mH) = x_{1}^{-1}x_{2}. \] 
Since $x_{1}^{-1}x_{2} \in C$, this implies $\mX \cup \mH = \{x_{1},\ldots, x_{m},h_{1},\ldots,h_{n}\} \subset C$.
Therefore we conclude $C = X$.

Next we consider the case  $\mX  \subset \langle x_{1} \rangle$. This happens only when $G=\Z = \langle g_{1} \rangle$ and $z_{G}= g_{1}^{N}$. Then we may choose $y$ so that $1< h_{1} \leq_{X} y$ holds, so $h_{1} \in C$. Then $x_{1}^{-1}h_{1} = \Delta_{H}g_{1}^{-1}h_{1} = h_{1}^{-1}z_{G} g_{1}^{-1}h_{1}$ so $z_{G}g_{1}^{-1} = g_{1}^{N-1} \in C$. This implies $z_{G}=z_{H} \in C$, so $C=X$.

\end{proof}

\subsection{Computational issues}
\label{sec:comp}

In this section we briefly mention the computational issue concerning the isolated ordering $<_{X}$. 
Let $G = \langle \mS \: | \:\mathcal{R} \rangle$ be a group presentation and $<_{G}$ be a left ordering of $G$. The {\em order-decision problem} for $<_{G}$ is the algorithmic problem of deciding for an element $g \in G$ given as a word on $\mS \cup \mS^{-1}$, whether $1 <_{G} g$ holds or not. 
Clearly, the order-decision problem is harder than the word problem, since $1 <_{G} g$ implies $1 \neq g$.
It is interesting to find an example of a left ordering $<_{G}$ of a group $G$, such that the order-decision problem for $<_{G}$ is unsolvable but the word problem for $G$ is solvable.

There is another algorithmic problem which is also related to the order-decision problem of isolated orderings.
We say a word on $\mG \cup \mG^{-1}$ is {\em $\mG$-definite} if $w$ is $\mG$-positive or $\mG$-negative, or empty. If $\mG$ defines an isolated ordering of $G$, then every $g \in G$ admits a $\mG$-definite word expression. The {\em $\mG$-definite search problem} is a problem to find a $\mG$-definite word expression of a given element of $G$.

 \begin{theorem}
 \label{thm:orderdecision}
 Let us take $G,H,X, <_{G},<_{H}, z_{G},z_{H},\mG,\mH,\mX$ as in Theorem \ref{thm:main}. 

\begin{enumerate}
\item The order-decision problem for $<_{X}$ is solvable if and only if the order-decision problems for $<_{G}$ and $<_{H}$ are solvable.
\item The $\{\mX,\mH\}$-definite search problem is solvable if and only if the $\mG$-definite search problem and the $\mH$-search problem are solvable.
\end{enumerate}
\end{theorem}

\begin{proof}
Since the restriction of $<_{X}$ to $G$ and $H$ yields the ordering $<_{G}$ and $<_{H}$ respectively, if the order-decision problem for $<_{X}$ is solvable, then so is for $<_{G}$ and $<_{H}$. Similarly, if $\{\mX,\mH\}$-definite search problem is solvable, then we are able to get $\mG$-positive (resp. $\mH$-positive) word by transforming $\{\mX,\mH\}$-positive word representing elements of $G$ (resp. $H$) by using $g_{i} =x_{i}\Delta_{H}$ so $\mG$-definite ($\mH$-definite) search problem is also solvable.

The proof of converse is implicit in the proof of Theorem \ref{thm:main} (i).
Recall that in the proof of Property C (Proposition \ref{prop:PropertyC}), we have shown that for a reduced standard factorization $\mF(x)= rp_{1}q_{1}\cdots p_{l}q_{l}$, $x>_{X} 1$ if $r \geq_{H}1$ and $x<_{X} 1$ if $r_{H}<1$. Moreover, the proof of Property C (Proposition \ref{prop:PropertyC}) is constructive, hence we can algorithmically compute an $\{ \mX,\mH\}$-negative word expression of $x$ if $r<_{H} 1$ if the $\mG$-definite search problem and the $\mH$-search problem is solvable.

Thus, to solve the order-decision problem or $\{ \mX,\mH\}$-definite search problem, it is sufficient to compute a reduced standard factorization.
We have established two different methods to compute a reduced standard factorization, in the proof of Lemma \ref{lem:stand} and Lemma \ref{lem:reducing}. Both proofs are constructive, hence we can algorithmically compute a reduced standard expression.
\end{proof} 

It is not difficult to analyze the computational complexity of order-decision problem or the $\{ \mX,\mH\}$-definite search problems based on the algorithm obtained from the proof of Proposition \ref{prop:PropertyC}, Lemma \ref{lem:stand} and Lemma \ref{lem:reducing}. In particular, we observe the following results.

\begin{proposition}
 Let us take $G,H,X, <_{G},<_{H}, z_{G},z_{H},\mG,\mH,\mX$ as in Theorem \ref{thm:main}. 
\begin{enumerate}
\item If the order-decision problems for $<_{G}$ and $<_{H}$ are solvable in polynomial time with respect to the length of the input of words, then the order-decision problem for $<_{X}$ is also solvable in polynomial time.
\item If the $\mG$-definite search problem and the $\mH$-definite search problem are solvable in polynomial time, then the $\{\mX,\mH\}$-definite search problem is also solvable in polynomial time. 
\item Moreover, if one can always find a $\mG$-definite and an $\mH$-definite word expression whose length are polynomial with respect to the length of the input word, then one can always find an $\{\mX,\mH\}$-definite word expression whose length is polynomial with respect to the length of the input word.
\end{enumerate}
\end{proposition}

\section{Examples}
\label{sec:exam}

In this section we give examples of isolated left orderings produced by Theorem \ref{thm:main}.
All examples in this section are new, and have various properties which previously known isolated orderings do not have. For the sake of simplicity, in the following examples we only use the infinite cyclic group $\Z$, the most fundamental example of group having isolated orderings, as the basic building blocks. 

Other group with isolated orderings, such as groups having only finitely many left-orderings, or the braid group $B_{n}$ with the Dubrovina-Dubrovin ordering $<_{DD}$, also can be used to construct new examples of isolated orderings. 

\subsection{Group having many distinct isolated orderings}

Let $a_{1},\ldots,a_{m}$ $(m>1)$ be positive integers bigger than one and consider the group obtained as a central cyclic amalgamated free product of $m$ infinite cyclic groups $\Z^{(i)} = \langle x_{i} \rangle$.
\begin{eqnarray*} 
G = G_{a_{1},\ldots, a_{m}} & = & *_{\Z} \Z^{(i)} \\
 & = & \left\langle x_{1},\ldots, x_{m}\: | \: x_{1}^{a_{1}}=x_{2}^{a_{2}}= \cdots =x_{m}^{a_{m}} \right\rangle 
\end{eqnarray*}

Recall that an infinite cycle group $\Z$ have exactly two left orderings, the standard one and its opposite. Using the standard left ordering for each factor $\Z^{(i)}$, by Theorem \ref{thm:main} we are able to construct an isolated left ordering $<_{G}$ so that the restriction of $<_{G}$ to the $i$-th factor $\Z^{(i)}$ is the standard left ordering.

First we give a detailed exposition of $<_{G}$ for the case $m=2$ and $m=3$.

\begin{example}
\label{exam:ZZZ}
{}$\;$\\
\begin{enumerate}
\item[(i)] First we begin with the case $m=2$, which was already considered in \cite{i1},\cite{n2}:
\[ G_{a_{1},a_{2}}= \Z^{(1)}*_{\Z}\Z^{(2)}= \langle x_{1},x_{2} \: | \: x_{1}^{a_{1}}=x_{2}^{a_{2}}\rangle \]

By Theorem \ref{thm:main}, we get an isolated ordering $<_{G}$ defined by the generating set $\{ x_{1}x_{2}^{1-a_{2}}, x_{2} \}$.

\item[(ii)] Next we consider the case $m=3$. There are two different ways to express $G$ as an amalgamated free products of $\Z$.

\begin{enumerate}
\item First we regard $G_{a_{1},a_{2},a_{3}} = G_{a_{1},a_{2}}*_{\Z} \Z^{(3)} = (\Z^{(1)}*_{\Z}\Z^{(2)})*_{\Z}\Z^{(3)}$. 

By (1), $G_{a_{1},a_{2}}$ have an isolated ordering defined by $\{ x_{1}x_{2}^{1-a_{2}}, x_{2} \}$. By applying Theorem \ref{thm:main} again, we get the isolated ordering $<_{(\bullet \bullet )\bullet }$ defined by $\{ x_{1}x_{2}^{1-a_{2}}x_{3}^{1-a_{3}}, x_{2}x_{3}^{1-a_{3}},x_{3}\}$.

\item Next we regard $G_{a_{1},a_{2},a_{3}} = \Z^{(1)}* G_{a_{2},a_{3}} = \Z^{(1)}*(\Z^{(2)}*_{\Z}\Z^{(3)})$. By applying Theorem \ref{thm:main}, we get the isolated ordering $<_{\bullet(\bullet\bullet)}$ defined by
$\{x_{1}x_{3}^{-a_{3}}x_{2}x_{3}^{1-a_{3}} = x_{1}x_{2}^{1-a_{2}}x_{3}^{1-a_{3}}, x_{2}x_{3}^{1-a_{3}},x_{3}\}$.\\
\end{enumerate}

Thus two orderings $<_{(\bullet \bullet )\bullet}$ and  $<_{\bullet(\bullet\bullet)}$ derived from different factorizations are the same ordering.
\end{enumerate}
\end{example}

As Example \ref{exam:ZZZ} (ii) suggests, the isolated orderings constructed from Theorem \ref{thm:main} are independent of the way of factorization as amalgamated free products, that is, the way of putting parenthesis in the expression $\Z^{(1)}*_{\Z} \Z^{(2)} *_{\Z} \cdots *_{\Z} Z^{(m)}$. All factorizations give the same isolated ordering $<_{G}$ defined by $\{s_{1},\ldots,s_{m}\}$, where $s_{i}$ is given by
\[ s_{i} = x_{i}x_{i+1}^{1-a_{i+1}} \cdots x_{m}^{1-a_{m}}. \]

This is checked by induction on $m$. Take a factorization of $G$ as
$G= G_1*_{\Z}G_2 = G_{a_{1},\ldots,a_{k}} *_{\Z} G_{a_{k+1},\ldots,a_{m}}$. By induction, the isolated ordering $<_{1}$ of $G_{1}$ is independent of a choice of a factorization of $G_{1}$, and is defined by
\[ s'_{i}=  x_{i}x_{i+1}^{1-a_{i+1}} \cdots x_{k}^{1-a_{k}} \;\;\;(i=1,\ldots, k). \]
Similarly, the isolated ordering $<_{2}$ of $G_{2}$ is independent of a choice of a factorization of $G_{2}$, and is defined by 
\[ s''_{j} = x_{j}x_{j+1}^{1-a_{j+1}} \cdots x_{m}^{1-a_{m}} \;\;\;(j=k+1,\ldots, m). \]
Thus by Theorem \ref{thm:main}, we get an isolated ordering $<_{G}$ of $G$ defined by 
\begin{eqnarray*}
 s_{i}&  =  &\left\{ \begin{array}{ll}
 s'_{i} x_{k+1}^{-a_{k+1}} x_{k+1}x_{k+2}^{1-a_{j+2}} \cdots x_{m}^{1-a_{m}} & (i=1,\ldots, k), \\
  x_{i}x_{i+1}^{1-a_{i+1}} \cdots x_{m}^{1-a_{m}} & (i=k+1,\ldots ,m)
  \end{array}
  \right. \\
  & = & x_{i}x_{i+1}^{1-a_{i+1}} \cdots x_{m}^{1-a_{m}} .
\end{eqnarray*}

The group $G$ is the simplest example of groups with isolated orderings constructed by Theorem \ref{thm:main}. Nevertheless the group $G$ and its isolated ordering $<_{G}$ have various interesting properties which have not appeared in the previous examples:\\

{\em (1): The isolated ordering $<_{G}$ of $G$ is not derived from  Dehornoy-like orderings if $G$ is not generated by two elements.}\\

As we mentioned earlier, the special kind of left-orderings called {\em Dehornoy-like orderings} produces isolated orderings, and all previously known examples of genuine isolated orderings are derived from Dehornoy-like orderings. 

In \cite{i1} it is proved that an isolated ordering derived from Dehornoy-like orderings has a lot of convex subgroups: if the isolated orderings $<_{H}$ of a group $H$ is derived from the Dehornoy-like orderings, then there are at least $r(<_{H})-1$ proper, $<_{H}$-convex nontrivial subgroups. On the other hand Theorem \ref{thm:main} (vi) shows the isolated orderings $<_{G}$ has only one proper, $<_{G}$-convex nontrivial subgroup.

If $G$ is not generated by two elements, then $r(<_{G})>2$. 
 This implies that the isolated ordering $<_{G}$ of $G$ is not derived from a Dehornoy-like ordering. This provides a counter example of somewhat optimistic conjecture: every genuine isolated ordering is derived from Dehornoy-like ordering. (Recall that all previously known examples of genuine isolated orderings are constructed by Dehornoy-like orderings.)

We remark that it is known that the group $G= G_{a_{1},\ldots,a_{m}}$ is a two-generator group if and only if $a_{i}$ and $a_{j}$ are not coprime for some $i \neq j$ \cite{mps}. Therefore for example, the isolated ordering of  $G_{2,3,4}$ in Example \ref{exam:ZZZ} (ii) is an isolated ordering which is not derived from a Dehornoy-like ordering.\\

{\em (2): The natural right $G$-action on $\LO(G)$ has at least $2(m-1)!$ distinct orbits derived from isolated orderings.}\\

There is a natural, continuous right $G$-action on $\LO(G)$, defined as follows: For a left ordering $<$ of $G$ and $g \in G$, we define the left ordering $< \! \cdot  g$ by $h\, (<\!\cdot g)\, h'$ if $hg < h'g$. This action sends an isolated ordering to an isolated ordering. Although this action is natural and important, little is known about the quotient $\LO(G) \slash G$. 

Recall that $G$ is written as the amalgamated free products of $m$ infinite cyclic groups $\Z^{(i)}$. As we have seen, the way of decomposition of $G$ ( the way of putting parenthesis) does not affect the obtained isolated ordering $<_{G}$.

On the other hand, for a permutation $\sigma \in S_{m}$, $G_{a_{1},\ldots,a_{m}} = G_{a_{\sigma(1)}, \ldots, a_{\sigma(m)}}$.
By viewing $G_{a_{1},\ldots,a_{m}} = G_{a_{\sigma(1)}, \ldots, a_{\sigma(m)}}$ and applying the construction above, we get an isolated ordering $<_{\sigma}$ whose minimal positive element is 
\[  x_{\sigma(1)} x_{\sigma(2)}^{1-a_{\sigma(2)}} \cdots x_{\sigma(m)}^{1-a_{\sigma(m)}} = x_{1}^{-(m-1)a_{1}} x_{\sigma(1)} x_{\sigma(2)}\cdots x_{\sigma(m)}. \]
Thus, for two permutations $\sigma$ and $\tau$, if 
$ x_{\sigma(1)} x_{\sigma(2)}\cdots x_{\sigma(m)}$ and $x_{\tau(1)} x_{\tau(2)}\cdots x_{\tau(m)}$ are not conjugate, then two isolated orderings $<_{\sigma}$ and $<_{\tau}$ belong to distinct $G$-orbits.
Hence, we are able to construct $(m-1)!$ distinct $G$-orbits of isolated orderings.

Recall that these orderings are constructed from the standard left orderings of $\Z^{(i)}$. By using the opposite of the standard left-ordering of $\Z$ instead, we get other $(m-1)!$ distinct $G$-orbits of isolated orderings in a similar way. Thus we have at least $2(m-1)!$ different $G$-orbits derived from isolated orderings. \\

{\em (3): The natural right $\textrm{Aut}(G)$-action on $\LO(G)$ has at least $(m-1)!$ distinct orbits derived from isolated orderings if all $a_{1},\ldots,a_{m}$ are distinct.}\\

As in the group $G$ itself, there is a natural right $\textrm{Aut}(G)$-action on $\LO(G)$. For a left ordering $<$ of $G$ and $\theta \in \textrm{Aut}(G)$, we define the left ordering $<\! \cdot \theta$ by $h < \!\!\cdot\theta\;  g $ if $h\theta < g\theta$.
The right $G$-action on $\LO(G)$ can be regarded as the restriction of the natural $\textrm{Aut}(G)$-action to the subgroup $\textrm{Inn}(G)$.
 
There is one symmetry which reduces the number of orbits: the involution defined by $x_{i} \mapsto x_{i}^{-1}$ $(i=1,\ldots,m)$. This amounts to taking the opposite ordering. If all $a_{1},\ldots,a_{m}$ are distinct, $\phi(a_{i}) \neq a_{j}^{\pm 1}$ for any $\phi \in \textrm{Aut}(G)$. Hence by a similar argument as (2), by looking at the minimal positive elements, we show that there are $(m-1)!$ distinct $\textrm{Aut}(G)$-orbit derived from isolated orderings. 
\\

Thus, the properties (2) and (3) show that the group $G$ has quite a lot of essentially different isolated orderings.

\subsection{Centerless group with isolated ordering}

Next we consider the construction of the case $z_{H}$ is non-central.
First of all, let $G_{m,n} = \langle b,c \: | \: b^{m} =c^{n} \rangle$. By Example \ref{exam:ZZZ} (i), $G_{m,n}$ has an isolated left ordering $<_{m,n}$ which is defined by $\{ bc^{1-n}, c\}$.

Let us consider a non-central element 
$bc= bc^{1-n} \cdot b^{m}$. Then it satisfies the inequality $b^{m} <_{m,n} bc <_{m,n} b^{2m}$. Since $b^{m}$ is $<_{m,n}$-cofinal central element, this shows that $bc$ is also $<_{m,n}$-positive cofinal.

$<_{m,n}$ is a $(bc^{1-n})$-right invariant ordering  by Lemma \ref{lem:minimal}, and $<_{m,n}$ is also a $b^{m}$-right invariant ordering  since $b^{m}$ is central. Thus, $<_{m,n}$ is a $(bc)$-right invariant ordering.

Thus, we can take the non-central element $bc$ as an element $z_{H}$ in Theorem \ref{thm:main} and we are able to apply the partially central cyclic amalgamation construction.
Now we consider the group $H =H_{p,q,m,n} = \Z*_{\Z} G_{m,n} =\Z *_{\Z} (\Z *_{\Z} \Z)$ defined by
\[  \langle a, b, c \: | \: b^{m} =c^{n}, a^{p} = (bc)^{q} \rangle. \]

This group has an isolated left ordering $<_{H}$, defined by $\{a(bc)^{1-q}, bc^{1-n}, c\}.$
Let us put $x=a(bc)^{1-q}$, $y=(bc)^{1-n}$, and $z=c$. Then the group $H_{p,q,m,n}$ is presented as
\[ H_{p,q,m,n} = \langle x,y,z \: | \: (yz^{n-1})^{m} = z^{n}, (x (yz^{n})^{q-1})^{p} = (yz^{n})^{q} \rangle\]
by using the generator $\{x,y,z\}$.

Clearly, $H$ has trivial center. This gives a first example of centerless group having isolated orderings. In fact, Theorem \ref{thm:main} will allow us to construct many examples of centerless group having isolated orderings.

\end{document}